\documentclass[12pt]{article}

\usepackage{graphicx}
\usepackage{latexsym,amssymb}

\usepackage{amsthm}
\usepackage{indentfirst}
\usepackage{amsmath}
\usepackage{hyperref}

\newtheorem{theoremalph}{Theorem}

\newtheorem{Theorem}{Theorem}[section]
\newtheorem*{Theorem A}{Theorem A}

\newtheorem{Definition}[Theorem]{Definition}
\newtheorem{Proposition}[Theorem]{Proposition}
\newtheorem{Lemma}[Theorem]{Lemma}

\newtheorem*{Remark}{Remark}
\newtheorem{Corollary}[Theorem]{Corollary}

\newtheorem*{Claim}{Claim}
\newtheorem*{Acknowledgements}{Acknowledgements}

 \def\NN{{\mathbb N}} 

 \def\RR{{\mathbb R}}

 \def\ZZ{{\mathbb Z}}

\def\dim{\operatorname{dim}}

\begin{document}

\title{On Pesin's entropy formula for dominated splittings without mixed behavior}

\author{Dawei Yang\and Yongluo Cao\footnote{D. Yang and Y. Cao would like to thank the support of NSFC 11125103, NSFC 11271152 and A Project Funded by the Priority Academic Program Development of Jiangsu Higher Education Institutions(PAPD). Y. Cao is the corresponding author.}}

\date{}

\maketitle

\begin{abstract}
For $C^1$ diffeomorphisms, we prove that  the Pesin's entropy formula holds for some invariant measure supported on any topological attractor that admits a dominated splitting without mixed behavior. We also prove Shub's entropy conjecture for diffeomorphisms having such kind of splittings.

\end{abstract}

\section{Introduction}

Pesin's entropy formula characterize the relationship between the metric entropy and Lyapunov exponents: the metric entropy is the integration of the sum of positive Lyapunov exponents. Sometimes, a measure that satisfies the Pesin's entropy formula is called an \emph{SRB} measure when there is at least one positive Lyapunov exponent. We would like to know the existence of measures that satisfy the entropy formula for a given system. Lots of results were got for $C^2$ maps. Since the absence of distortion bounds, we lose some method to get SRB measures for $C^1$ maps. However, there are results for $C^1$ maps. See \cite{CaQ01,CCE15,Qiu11,SuT12,Tah02} for instance.

In this paper, we consider a topogical attractor which admits a dominated splitting without mixed behavior. We show the existence of measures  satisfying Pesin's entropy formula for this kind of systems. Such a splitting is satisfied in some natural setting, for instance, if a non-periodic transitive set of a surface diffeomorphism has a non-trivial dominated splitting, then this dominated splitting has no mixed behavior.

Let $f$ be a diffeomorphism on a manifold $M$ whose dimension is $d$. For a compact invariant set $\Lambda$, one says that $\Lambda$ admits a \emph{dominated splitting} if there is a continuous invariant splitting $T_\Lambda M=E\oplus F$, and constants $C>0$ and $\lambda\in(0,1)$ such that for any $x\in\Lambda$, and $n\in\NN$, any $u\in E(x)\setminus\{0\}$ and any $v\in F(x)\setminus \{0\}$, we have
$$\frac{\|Df^n(u)\|}{\|u\|}\le C\lambda^n\frac{\|Df^n(v)\|}{\|v\|}.$$

We say a dominated splitting $T_\Lambda M=E\oplus F$ has \emph{no mixed behavior} if for any measure $\mu$ supported on $\Lambda$, every Lyapunov exponent of $\mu$ along $E$ is non-positive and every Lyapunov exponent of $\mu$ along $F$ is non-negative. Equivalently, we have that 
$$\liminf_{n\to\infty}\frac{1}{n}\log\|Df^n|_{E(x)}\|\le 0,~~~\liminf_{n\to\infty}\frac{1}{n}\log\|Df^{-n}|_{F(x)}\|\le 0,~~~\forall x\in\Lambda.$$

\begin{theoremalph}\label{Thm:formula}

For a $C^1$ diffeomorphism $f$, if an attractor $\Lambda$ admits a dominated splitting $T_\Lambda M=E\oplus F$ without mixed behavior, then there is a measure $\mu$ supported on $\Lambda$ satisfying Pesin's entropy formula.

\end{theoremalph}

In a recent paper by Liu and Lu \cite{LiL15}, for a $C^2$ map, they got measures satisfying Pesin's entropy formula for a topological attractor which admits a partially hyperbolic splitting without mixed behavior.  Cowieson and Young proved the existence of SRB measures \cite[Corollary 1]{CoY05} if $\Lambda$ is an attractor of a $C^\infty$ diffeomorphism $f$ and $\Lambda$ admits a dominated splitting $T_\Lambda M=E\oplus F$ without mixed behavior and $\limsup_{n\to\infty}(1/n)\log\|Df^n|_{F(x)}\|>0$ for any point $x\in\Lambda$.

With some additional effort from the proof of Theorem~\ref{Thm:formula}, we can know that the topological entropy varies upper semi continuous w.r.t. the diffeomorphisms. Thus, by a usual argument we can know the entropy conjecture is also true for dominated splittings without mixed behavior.

The diffeomorphism $f$ induces naturally a map $f_{*,k}:H_k(M,R)\to H_k(M,R)$ for any $0\le k\le d$, where $H_k(M,R)$ is the $k$-th homology group of $M$. Shub conjectured in \cite{Shu74} that for every $C^1$ diffeomorphism $f$,
$$\max_{0\le i\le d}{\rm sp}(f_{*,i})\le {\rm h}_{top}(f),$$
where ${\rm sp}(A)$ is the spectral radius of a linear map $A$.
\begin{theoremalph}\label{Thm:conjecture}
For a $C^1$ diffeomorphism $f$, if $M$ admits a dominated splitting without mixed behavior, then the entropy conjecture is true, i.e.,
$$\max_{0\le i\le d}{\rm sp}(f_{*,i})\le {\rm h}_{top}(f).$$

\end{theoremalph}

Shub's entropy conjecture is still open. However, there are lots of interesting results on that. We give a partial list:
\begin{itemize}

\item \cite{Yom87} proved that Shub's conjecture holds for $C^\infty$ maps.

\item \cite{RuS75,ShW75} proved the conjecture for Anosov systems and general Axiom A diffeomorphisms.

\item \cite{Man75} proved the conjecture for the three-dimensional case.

\item \cite{SaX10} proved the conjecture for partially hyperbolic systems with one-dimensional center bundle.

\item \cite{LVY13} proved the conjecture for diffeomorphisms that away from ones with a homoclinic tangency.

\item \cite{LiL15} proved the conjecture for diffeomorphisms admits a partially hyperbolic splitting without mixed behavior.

\end{itemize}

We notice that the assumption of Theorem~\ref{Thm:conjecture} is not contained in any result listed above.

\bigskip

We will also consider the properties of asymptotically entropy expansive and principal symbolic extension in Section~\ref{Sec:esti-localentropy}.

\begin{Acknowledgements}
We would like to thank our student Yuntao Zang for helping us to improve the first version a lot.

\end{Acknowledgements}

\section{Definitions and Properties of entropies}

In this section, we give the definitions and properties of metric entropy, local entropy and topological entropy.

\subsection{Metric entropies}

Let $\mu$ be a probability measure. For a finite measurable partition ${\cal B}=\{B_1,B_2,\cdots,B_k\}$, we define
$${\rm H}_\mu({\cal B})=\sum_{i=1}^k-\mu(B_i)\log\mu(B_i),$$

and

$$\bigvee_{i=0}^{n-1}f^{-i}({\cal B})=\{C:~C=\bigcap_{i=0}^{n-1}f^{-i}(B_{i_j})\}.$$

If $\mu$ is an invariant probability measure of a map $f$, the metric entropy of $\mu$ w.r.t. a partition ${\cal B}$ is
$${\rm h}_\mu(f,{\cal B})=\lim_{n\to\infty}\frac{1}{n}{\rm H}_{\mu}(\bigvee_{i=0}^{n-1}f^{-i}({\cal B})),$$
and the metric entropy of $\mu$ is
$${\rm h}_\mu(f)=\sup_{{\cal B}:~\textrm{partition}}{\rm h}_{\mu}(f,{\cal B}).$$

\begin{Definition}
Given a finite partition ${\cal B}=\{B_1,B_2,\cdots,B_n\}$, the \emph{norm} of the partition is $\max_{1\le i\le n}{\rm Diam}(B_i)$. The norm of $\cal B$ is denoted by $\|{\cal B}\|$.

Given a measure $\mu$, a partition ${\cal B}$ is called \emph{regular} if  $\mu(\partial B)=0$ for any $B\in\cal B$; it is called \emph{$\alpha$-regular} if $\|{\cal B}\|<\alpha$ and it is regular.
\end{Definition}

By the definition, we have the following lemma:

\begin{Lemma}\label{Lem:finite-approximation}
Given a regular partition ${\cal B}$ of a measure $\mu$ of a diffeomorphism $f$, and given $n\in\NN$, for any $\varepsilon>0$, there is $\delta>0$ such that for any $g$ which is $\delta$-$C^1$-close to $f$, for any invariant measure $\nu$ of $g$ which is $\delta$-close to $\mu$ in the weak-$*$ topology, then
$$|\frac 1n {\rm H}_{\mu}(\bigvee_{i=0}^{n-1}f^{-i}({\cal B})) - \frac 1n {\rm H}_{\nu}(\bigvee_{i=0}^{n-1}g^{-i}({\cal B}))|<\varepsilon.$$

\end{Lemma}

The following fundamental results are from \cite[Section 8.2]{Wal82}:

\begin{Lemma}\label{Lem:H}

Let $\mu_1,\mu_2,\cdots,\mu_n$ be probability measures and $s_1,s_2,\cdots,s_n$ be non-negative numbers such that $\sum s_i=1$. For any partition ${\cal B}$, we have
$$\sum_{i=1}^n s_i{\rm H}_{\mu_i}({\cal B})\le {\rm H}_{\sum_{i=1}^n s_i\mu_i}({\cal B}).$$

\end{Lemma}

\subsection{Local entropy}

We need to define the Bowen balls or dynamical balls in the entropy theory. Given a point $x$ and $\alpha>0$,
\begin{itemize}

\item the closed ball of radius $\alpha$ at $x$: $B(x,\alpha)=\{y\in M:~d(x,y)\le\alpha\}$;

\item $n$-th Bowen ball for $f$: $B_n(x,\alpha,f)=\bigcap_{0\le i\le n-1}f^{-i}(B(f^i(x),\alpha))$; for simplicity, we denote $B_n(x,\alpha)=B_n(x,\alpha,f)$ if there is no confusion;

\item bi-$n$-th Bowen ball: $B_{\pm n}(x,\alpha)=\bigcap_{-n+1\le i\le n-1}f^{-i}(B(f^i(x),\alpha))$;

\item infinite Bowen ball for $f$: $B_\infty(x,\alpha)=B_{+\infty}(x,\alpha)=\bigcap_{n\in\NN}f^{-n}(B(f^n(x),\alpha))$;  for simplicity, we denote $B_\infty(x,\alpha)=B_\infty(x,\alpha,f)$ if there is no confusion;

\item bi-infinite Bowen ball: $B_{\pm \infty}(x,\alpha)=\bigcap_{n\in\ZZ}f^{-n}(B(f^n(x),\alpha))$

\end{itemize}

\begin{Definition}[Local entropy]
For a compact set $\Gamma$ (not necessarily invariant), for $n\in\NN$ and $\delta$, a finite set $P\subset \Gamma$ is called an $(n,\delta)$-spanning set for $f$ (or $(n,\delta,f)$-spanning set) if $P\cap B_n(x,\delta)$ is not empty for any $x\in\Gamma$. The minimal cardinality of all $(n,\delta)$-spanning set is denoted by $r_n(\Gamma,\delta)$.

Then one can define the entropy of $\Gamma$ by
$${\rm h}(f,\Gamma)={\rm h}(f|_\Gamma)=\lim_{\delta\to 0}\limsup_{n\to\infty} \frac{1}{n} \log r_n(\Gamma,\delta).$$

When $\Gamma$ is a compact invariant set, we also call ${\rm h}(f|_\Gamma)$ the topological entropy of $f$ on $\Gamma$.  Sometimes, one denotes it by ${\rm h}_{top}(f|_\Gamma)$.

We then define the local entropy of the scale $\alpha$ for a compact set $\Gamma$ by
$${\rm h}_{\alpha}(f|_\Gamma)=\sup_{x\in\Gamma}{\rm h}(f,B_{\infty}(x,\alpha)).$$

\end{Definition}

One has the following lemma for spanning sets from Bowen \cite[Lemma 2.1]{Bow72}.
\begin{Lemma}\label{Lem:product-spanning}
Assume that $\Gamma$ is a compact set and $\varepsilon>0$. Let $0=t_0<t_1<t_2<\cdots<t_r=n$ be integers. If $P_i$ is a $(t_{i+1}-t_i,\varepsilon)$-spanning set of $f^{t_i}(\Gamma)$ for any $0\le i\le r-1$, then
$$r_n(\Gamma,2\varepsilon)\le\prod_{i=0}^{r-1}\# P_i.$$

\end{Lemma}

By using the definition, we have

\begin{Lemma}\label{Lem:entropy-nStep}
Given any $\alpha>0$, for any $x\in M$ and any $m\in\NN$, we have
$${\rm h}(f,B_{\pm\infty}(x,\alpha))={\rm h}(f,B_{\pm\infty}(f^m(x),\alpha)).$$
\end{Lemma}

\begin{proof}

For any $\varepsilon>0$, let us fix an $\varepsilon/4$-dense set in $M$ whose cardinality is $N_\varepsilon$. Thus for any compact set $\Gamma$, there is a $(1,\varepsilon)$-spanning set whose cardinality is at most $N_\varepsilon$.  For any $n\in\NN$, by Lemma~\ref{Lem:product-spanning}, we have
$$r_{m+n}(B_{\pm\infty}(x,\alpha),2\varepsilon)\le N_\varepsilon^m r_n(B_{\pm\infty}(f^m(x),\alpha),\varepsilon).$$
On the other hand, for any $\varepsilon>0$ and any $n\in\NN$, if $P_{m+n}$ is an $(n+m,\varepsilon)$-spanning set for $f$ of $B_{\pm\infty}(x,\alpha)$ satisfying $\# P_{m+n}=r_{m+n}(B_{\pm\infty}(x,\alpha),\varepsilon)$, then $f^m(P_{m+n})$ is an $(n,\varepsilon)$-spanning set for $f$ of $B_\infty(f^m(x),\alpha)$. Hence, we have
$$r_{m+n}(B_{\pm\infty}(x,\alpha),\varepsilon)=\# f^m(P_{m+n})\ge r_n(B_{\pm\infty}(f^m(x),\alpha),\varepsilon).$$ 
By taking the limits, one can get the conclusion. 
\end{proof}

\subsection{Local entropies for $f$ and $f^{-1}$}

In this subsection, we need to prove the following proposition. We borrow some ideas from \cite[Proposition 2.5]{LVY13}.

\begin{Proposition}\label{Pro:inverse-entropy}
For any ergodic measure $\mu$, there is a full $\mu$-measure set $R$ such that
$$\sup_{x\in R}{\rm h}(f,B_{\pm\infty}(x,\alpha))=\sup_{x\in R}{\rm h}(f^{-1},B_{\pm\infty}(x,\alpha)).$$

\end{Proposition}

\begin{proof}
In fact, since $\mu$ is ergodic, one can take
$$R=\{x:~\frac{1}{n}\sum_{i=0}^{n-1}\delta_{f^i(x)}\to\mu,~\frac{1}{n}\sum_{i=0}^{n-1}\delta_{f^{-i}(x)}\to\mu\}.$$

In this proposition, the situation for $f$ and $f^{-1}$ is symmetric. Without loss of generality, one can assume that there is a point $x_0\in R$ such that
$${\rm h}(f,B_{\pm\infty}(x_0,\alpha))>\sup_{x\in R}{\rm h}(f^{-1},B_{\pm\infty}(x,\alpha)).$$
Thus, one can find two numbers $a_1>a_2$ such that
$${\rm h}(f,B_{\pm\infty}(x_0,\alpha))>a_1>a_2>\sup_{x\in R}{\rm h}(f^{-1},B_{\pm\infty}(x,\alpha)).$$
Recall the definition of the local entropy, there is $\varepsilon_0>0$ small such that for any $\varepsilon\in(0,\varepsilon_0)$, we have
$$\limsup_{n\to\infty}\frac{1}{n}\log r_n(B_{\pm\infty}(x_0,\alpha),\varepsilon)>a_1.$$
In other words, there is a sequence of integers $\{n_i\}$ such that
$$r_{n_i}(B_{\pm\infty}(x_0,\alpha),\varepsilon)>{\rm e}^{a_1 n_i}.$$
\bigskip

For this $\varepsilon_0>0$, we choose a finite set $P_0$ that $\varepsilon_0/8$-dense in $M$. Thus, for any compact subset $\Gamma$ of $M$, there is a $\varepsilon_0/2$-dense set in $\Gamma$, whose cardinality is at most $\# P_0$.

\bigskip

Take 
$$\mu_{n_i}=\frac{1}{n}\sum_{j=0}^{n-1}\delta_{f^j(x_0)}.$$
Since $\mu$ is ergodic, we have that $\mu_{n_i}\to\mu$ as $i\to\infty$.

For $\mu$, for each $n\in\NN$, one can find $\varepsilon_n\ll\varepsilon_0$ such that if we define the set $R_n$ as
$$R_n=\{x\in R:~r_m(B_{\pm\infty}(x,\alpha),\varepsilon_n/4,f^{-1})<{\rm e}^{a_2 m},~~~\forall m\ge n\},$$
then we have
\begin{itemize}
\item$\mu(\cup_{n\in\NN}R_n)=1$. 
\item $\{R_n\}$ is an increasing sequence of measurable sets.
\end{itemize}

Then one can choose an increasing sequence of \emph{compact} sets $\{\Lambda_n\}$ such that
\begin{itemize}

\item $\Lambda_n\subset R_n$ for each $n\in\NN$.

\item $\mu(\cup_{n\in\NN}\Lambda_n)=1$.

\end{itemize}

Now we fix some $n$ that is probably large enough. For any $x\in\Lambda_n$, let $P_n(x)$ be an $(n,\varepsilon_n/4,f^{-1})$-spanning set of $B_{\pm\infty}(x,\alpha)$ such that $\# P_n(x)<{\rm e}^{a_2 n}$. Then
$$U_n(x)=\bigcup_{z\in P_n(x)}B_n(z,\varepsilon_n/2,f^{-1})$$
is a neighborhood of $B_{\pm\infty}(x,\alpha)$.

We have the following observations.

\begin{itemize}

\item $f^{-n}P_n(x)$ is a $(n,\varepsilon_n/4)$-spanning set of $B_{\pm\infty}(f^{-n}(x),\alpha)$ for $f$ for any $x\in R_n$. It is clear that $\# f^{-n}P_n(x)<{\rm e}^{a_2 n}$.

\item $\mu(f^{-n}(\Lambda_n))=\mu(\Lambda_n)$

\end{itemize}

Now we choose a smaller neighborhood $V_n(x)\subset U_n(x)$ of $x$ and an integer $N_n(x)$ such that for any $y\in V_n(x)$, we have
$$B_{\pm N_n(x)}(y,\alpha)\subset U_n(x).$$
By the definition of $U_n(x)$, we have that for any $y\in V_n(x)$, $B_{\pm N_n(x)}(y,\alpha)$ is $(n,\varepsilon_n/2,f^{-1})$-spanned by $P_n(x)$.

\bigskip

As a corollary, we have that
$$\{V_n(x)\}_{x\in\Lambda_n}$$
is an open covering of $\Lambda_n$. Thus, there are finitely points $\{x_1,x_2,\cdots,x_k\}\subset\Lambda_n$ such that
$$W_n=\bigcup_{1\le j\le k}V_n(x_j)\supset\Lambda_n.$$
Consequently,
$$\lim_{n\to\infty}\mu(W_n)=1.$$
Take 
$$H(n)=\max\{n,N_n(x_1),\cdots,N_n(x_k)\}.$$

We have that 
$$\lim_{i\to\infty}\mu_{n_i}(f^{-n}(W_n))=\lim_{i\to\infty}\mu_{n_i}(f^{-n}(W_n))\ge\mu(f^{-n}(W_n))=\mu(W_n)\ge\mu(\Lambda_n).$$

Now we consider the positive iteration of $x_0$. For $n_i$ large, we find a sequence of times $0=\iota_0<\iota_1<\cdots<\iota_L=n_i$ by the following way inductively:
\begin{itemize}

\item if $f^{\iota_j}\in f^{-n}(W_n)$ and $\iota_j\in[H(n),n_i-H(n)]$, then $\iota_{j+1}=\iota_j+n$;

\item otherwise, one takes $\iota_{j+1}=\iota_j+1$.

\end{itemize}

Let 
$$A_n=\{\iota_j:~f^{\iota_j}(x_0)\in f^{-n}(W_n),~\iota_j\in[H(n),n_i-H(n)]\},~~~B_n=\{\iota_i\}_{0\le j\le L}\setminus A_n.$$
We have the following properties:
\begin{itemize}

\item if $\iota_j\in A_n$, then there is some $k_j\in\{1,2,\cdots,k\}$ such that $f^{\iota_j+n}(x_0)\in V(x_{k_j})$, and
$$f^{\iota_j+n}(B_{\pm n_i}(x_0,\alpha))\subset B_{\pm H(n)}(f^{\iota_j+n}(x_0),\alpha)\subset U_n(x_{k_j}).$$

\end{itemize}

This implies that $f^{\iota_j}(B_{\pm n_i}(x_0,\alpha))$ is $(n,\varepsilon_n/4,f)$-spanned by $f^{-n}P_n(x_{k_j})$

By Lemma~\ref{Lem:product-spanning}, we have

$$r_{n_i}(B_{\pm n_i}(x_0,\alpha),\varepsilon_0/2)\le\left(\prod_{\iota_j\in A_n}\# f^{-n}P_n(x_{k_j})\right)\cdot (\# P_0)^{\# B_n}\le {\rm e}^{a_2n\#A_n}\cdot (\# P_0)^{\#B_n},$$

By definitions, we have
\begin{itemize}

\item $n\#A_n\le n_i$,

\item $\#B_n\le \#\{0\le j<n_i:~f^j(x_0)\notin f^{-n}(W_n)\}+2H(n)\le (1-\mu_{n_i}(W_n))n_i+2H(n)$.

\end{itemize}

This implies that

\begin{eqnarray*}
r_{n_i}(B_{\pm n_i}(x_0,\alpha),\varepsilon_0/2)&\le&{\rm e}^{a_2 n_i}\cdot (\# P_0)^{(1-\mu_{n_i}(W_n))n_i+2H(n)}.
\end{eqnarray*}
Thus,
$$\frac{1}{n_i}\log r_{n_i}(B_{\pm n_i}(x_0,\alpha),\varepsilon_0/2)\le a_2 +[(1-\mu_{n_i}(W_n))+2\frac{H(n)}{n_i}]\log\# P_0.$$
For fixed $n$, we have that $H(n)$ is much smaller than $n_i$. By taking $n_i\to\infty$, we have
$$\limsup_{n_i\to\infty}\frac{1}{n_i}r_{n_i}(B_{\pm \infty}(x_0,\alpha),\varepsilon_0/2)\le\limsup_{n_i\to\infty}\frac{1}{n_i}r_{n_i}(B_{\pm n_i}(x_0,\alpha),\varepsilon_0/2)\le a_2+(1-\mu(W_n))\log\# P_0.$$

Then by asking $n\to\infty$,
$$\limsup_{n_i\to\infty}\frac{1}{n_i}r_{n_i}(B_{\pm \infty}(x_0,\alpha),\varepsilon_0/2)\le a_2+\lim_{n\to\infty}(1-\mu(W_n))\log\# P_0=a_2<a_1.$$
We get a contradiction. The proof is complete.

\end{proof}

In fact, we have the following more accurate characterization.

\begin{Proposition}\label{Pro:almostevery}
For any ergodic measure $\mu$, there is a constant $H$ such that for $\mu$-a.e. $x$, we have
$$H={\rm h}(f,B_{\pm\infty}(x,\alpha))={\rm h}(f^{-1},B_{\pm\infty}(x,\alpha)).$$

\end{Proposition}
\bigskip

For proving Proposition~\ref{Pro:almostevery}, we need to adapt the definition of the entropy. For a compact invariant set $\Gamma$, given $n\in\NN$ and $\varepsilon>0$, a subset $P$ of $\Gamma$ is called an $(n,\varepsilon)$-separated set of $\Gamma$ if for any $x,y\in P$, $d_n(x,y)>\varepsilon$. Denote by $s_n(\Gamma,\varepsilon)$ the largest cardinality for any $(n,\varepsilon)$-subset of  $\Gamma$.

By summarizing \cite[Chapter 7.2]{Wal82}, we have
\begin{itemize}
\item $r_n(\Gamma,\varepsilon)\le s_n(\Gamma,\varepsilon)\le r_n(\Gamma,\varepsilon_2)$ for any $n$ and any $\varepsilon$.

\item ${\rm h}(f,\Gamma)=\lim_{\varepsilon\to0}\limsup_{n\to\infty}(1/n)\log s_n(\Gamma,\varepsilon)=\lim_{\varepsilon\to0}\limsup_{n\to\infty}(1/n)\log r_n(\Gamma,\varepsilon)$.

\end{itemize}

We need to modify the definition of $s_n$ to $\bar s_n$ by the following way: for a compact invariant set $\Gamma$, given $n\in\NN$ and $\varepsilon>0$, a subset $P$ of $\Gamma$ is called an \emph{closed $(n,\varepsilon)$-separated set} of $\Gamma$ if for any $x,y\in P$, $d_n(x,y)\ge\varepsilon$. Denote by $\bar s_n(\Gamma,\varepsilon)$ the largest cardinality for any closed $(n,\varepsilon)$-separated set of  $\Gamma$. By using the definitions, we have the following properties:

\begin{itemize}
\item Given $\varepsilon>0$ and $n\in\NN$, we have
$$s_n(\Gamma,\varepsilon)\le \bar s_n(\Gamma,\varepsilon)\le s_n(\Gamma,\varepsilon/2).$$
\item ${\rm h}(f,\Gamma)=\lim_{\varepsilon\to0}\limsup_{n\to\infty}(1/n)\log \bar s_n(\Gamma,\varepsilon)$.

\end{itemize}
\bigskip

Now we can give the proof of Proposition~\ref{Pro:almostevery}.
\begin{proof}[Proof of Proposition~\ref{Pro:almostevery}]
For fixed $\alpha>0$, we define the local entropy function
$$H(x)={\rm h}(f,B_{\pm\infty}(x,\alpha)).$$
We need to verify that $H(x)$ is measurable. After that, by Lemma~\ref{Lem:entropy-nStep}, we have $H(x)=H(f(x))$, and then by the ergodicity of $\mu$, we have that $H(x)$ is constant for $\mu$-a.e. $x$. By the same reason, we have that ${\rm h}(f^{-1},B_{\pm\infty}(x,\alpha))$ is also a constant for $\mu$-a.e. $x$. Then by Proposition~\ref{Pro:inverse-entropy}, one can conclude this proposition.

To verify $H$ is a $\mu$-measurable, it is enough to check that given $\varepsilon>0$, the set 
$$L_a=\{x:~\limsup_{n\to\infty}\frac{1}{n}\log \bar s_n(B_{\pm\infty}(x,\alpha),\varepsilon)>a\}$$
is $\mu$-measurable for any $a>0$. We have that
$$L_a=\bigcup_{k\in\NN}\bigcap_{n\in\NN}\bigcup_{m\ge n}\{x:~\bar s_m(B_{\pm\infty}(x,\alpha),\varepsilon)\ge {\rm e}^{(a+1/k) m}\}.$$
Thus it is enough to show that the set
$$L_{a,m}=\{x:~\bar s_m(B_{\pm\infty}(x,\alpha),\varepsilon)\ge{\rm e}^{a m}\}$$
is $\mu$-measurable for any $a>0$ and $m\in\NN$. This can be deduced the fact that $L_{a,m}$ is closed by the upper semi continuity of bi-infinity Bowen balls.

In fact, assume that there is a sequence $\{x_n\}\subset L_{a,m}$ such that $\lim_{n\to\infty}x_n=x$, we need to show that $x\in L_{a,m}$. For each $x_n$, there is a closed $(m,\varepsilon)$-separated set $P_n=\{y_n^1,y_n^2,\cdots,y_n^{N_n}\}$ contained in $B_{\pm\infty}(x_n,\alpha)$ whose cardinality is $N_n=\# P_n\ge [{\rm e}^{a m}]$. 

By taking a subsequence if necessary, one can assume that for any $1\le j\le [{\rm e}^{a m}]$, $\lim_{n\to\infty}y_n^j=y^j\in B_{\pm\infty}(x,\alpha)$. Moreover, we have that for any $1\le i<j\le [{\rm e}^{a m}]$, $d_m(y^i,y^j)\ge \varepsilon$. This implies that there is a closed $(m,\varepsilon)$-separated set contained in $B_{\pm\infty}(x,\alpha)$ whose cardinality is at least $[{\rm e}^{am}]$, and hence ${\rm e}^{am}$. Consequently, $x\in L_{a,m}$. The proof is complete now.
\end{proof}

\section{Upper semi continuity of entropies}

%
%
%
%

\bigskip

In this section, we will mainly prove the upper semi continuity of the metric entropy w.r.t. invariant measures. Actually, we prove the following stronger result.
\begin{Theorem}\label{Thm:usc}
Assume that a compact invariant set $\Lambda$ admits a dominated splitting $T_\Lambda M=E\oplus F$ without mixed behavior. If there is a sequence of diffeomorphisms $\{f_n\}$ and  a sequence of invariant measures $\mu_n$ such that each $\mu_n$ is an invariant measure of $f_n$ and supported on a compact invariant set $\Lambda_n$ of $f_n$, and 
$$\lim_{n\to\infty}f_n=f,~~~\lim_{n\to\infty}\mu_n=\mu,~~~\limsup_{n\to\infty}\Lambda_n\subset\Lambda,$$
then
$$\limsup_{n\to\infty}{\rm h}_{\mu_n}(f_n)\le {\rm h}_{\mu}(f).$$

\end{Theorem}

We first give some consequences of Theorem~\ref{Thm:usc}, and then give its proof.

\subsection{Consequences of Theorem~\ref{Thm:usc}}

One says that the entropy function is \emph{upper semi continuous w.r.t. the measures} if for any measure $\mu$ and any sequence of measures $\mu_n$ such that $\lim_{n\to\infty}\mu_n=\mu$, then $\limsup_{n\to\infty}{\rm h}_{\mu_n}(f)\le {\rm h}_{\mu}(f)$.

\begin{Corollary}\label{Cor:usc}

Assume that a compact invariant set $\Lambda$ admits a dominated splitting $T_\Lambda M=E\oplus F$ without mixed behavior. Then the metric entropy is upper semi continuous w.r.t. the measures.

\end{Corollary}

The corollary can be deduced from Theorem~\ref{Thm:usc} directly.

\bigskip

The upper semi continuity of the entropy function can be applied in thermodynamical formalism. For any continuous function $\varphi$, the pressure of $\varphi$ is defined by
$$P(\varphi)=\sup_{\mu~\textrm{invariant}}\{{\rm h}_\mu(f)+\int \varphi{\rm d}\mu\}.$$
A measure $\mu$ is called an \emph{equilibrium state} of $\varphi$ if $P(\varphi)={\rm h}_\mu(f)+\int \varphi{\rm d}\mu$. By the upper semi continuity, we have the following corollary directly:
\begin{Corollary}\label{Cor}

Assume that a compact invariant set $\Lambda$ admits a dominated splitting $T_\Lambda M=E\oplus F$ without mixed behavior. Then every continuous function of $\Lambda$ has an equilibrium state on $\Lambda$.
\end{Corollary}

Another corollary is the upper semi continuity of topological entropy w.r.t. the diffeomorphisms.

\begin{Corollary}\label{Cor:usc-top}
Assume that a compact invariant set $\Lambda$ admits a dominated splitting $T_\Lambda M=E\oplus F$ without mixed behavior. If $f_n\to f$ as $n\to\infty$ in the $C^1$ topology and $\Lambda_{f_n}$ is a compact invariant set of $f_n$ satisfying $\limsup_{n\to\infty}\Lambda_{f_n}\subset\Lambda$, then 
$$\limsup_{n\to\infty}{\rm h}_{top}(f_n|_{\Lambda_{f_n}}) \le  {\rm h}_{top}(f |_{\Lambda_f}).$$

\end{Corollary}

\begin{proof}

For each $n$, we take an ergodic measure $\mu_n$ supported on $\Lambda_{f_n}$ such that
$${\rm h}_{\mu_n}(f_n|_{\Lambda_{f_n}})>{\rm h}_{top}(f_n|_{\Lambda_{f_n}})-\frac 1n.$$

By taking a subsequence if necessary, we assume that $\lim_{n\to\infty}\mu_n=\mu$ for some invariant measure $\mu$ of $f$. By Theorem~\ref{Thm:usc}, we have that
$${\rm h}_{top}(f|_{\Lambda_f})\ge {\rm h}_{\mu}(f|_{\Lambda_f}) \ge \limsup_{n\to\infty}{\rm h}_{\mu_n}(f_n|_{\Lambda_{f_n}})=\limsup_{n\to\infty}\left({\rm h}_{\mu_n}(f_n|_{\Lambda_{f_n}})+\frac 1n\right)=\limsup_{n\to\infty}{\rm h}_{top}(f_n|_{\Lambda_{f_n}}).$$

\end{proof}

\begin{proof}[Proof of Theorem~\ref{Thm:conjecture}]

Now we consider a diffeomorphism $f$ such that $M$ admits a dominated splitting without mixed behavior. There is a neighborhood $\cal U$ of $f$ such that any $g\in\cal U$ is isotropic to $f$. Thus we have
$$\max_{0\le i\le d}{\rm sp}(f_{*,i})= \max_{0\le i\le d}{\rm sp}(g_{*,i}),$$
For any $\varepsilon>0$, we choose a $C^\infty$ diffeomorphism $g\in\cal U$ such that by applying Yomdin's result \cite{Yom87}, we have
$${\rm h}_{top}(f)>{\rm h}_{top}(g)-\varepsilon\ge \max_{0\le i\le d}{\rm sp}(g_{*,i})-\varepsilon=\max_{0\le i\le d}{\rm sp}(f_{*,i})-\varepsilon.$$
Then by the arbitrariness of $\varepsilon$, one can complete the proof.

\end{proof}

\subsection{Uniformity on dominated splittings without mixed behavior}


\begin{Lemma}\label{Lem:uniformity}
Assume that $\Lambda$ admits a dominated splitting without mixed behavior. Then for any $\beta>0$, there is $N=N(\beta)\in\NN$ and a neighborhood $\cal U$ of $f$ such that for any $g\in\cal U$ and a neighborhood $U$ of $\Lambda$, for any compact invariant set $\Lambda_g$ of $g$ that is contained in $U$, we have that $\Lambda_g$ admits a dominated splitting 
$$T_{\Lambda_g}M=E_g\oplus F_g,$$
and 
$$\|Dg^N|_{E_g(x)}\|\le (1+\beta)^N,~~~\|Dg^{-N}|_{F_g(x)}\|\le (1+\beta)^N.$$

\end{Lemma}

\begin{proof}
By the main techniques in \cite{Cao03}, for $\beta/2$, there is $N>0$ such that for any $x\in\Lambda$, we have
$$\|Df^N|_{E(x)}\|\le (1+\beta/2)^N,~~~\|Df^{-N}|_{F(x)}\|\le (1+\beta/2)^N.$$
Thus there is a neighborhood $\cal U$ of $f$ such that for any $g\in\cal U$, if a compact invariant set $\Lambda_g$ of $g$ is contained in a small neighborhood of $\Lambda$, then $\Lambda_g$ has a dominated splitting $T_{\Lambda_g}M=E_g\oplus F_g$. By shrinking $\cal U$ and $U$ if necessary, we have that $E_g$ and $F_g$ are close to $E$ and $F$, respectively.
Thus for any $x\in\Lambda_g$, we have
$$\|Dg^N|_{E_g(x)}\|\le (1+\beta)^N,~~~\|Dg^{-N}|_{F_g(x)}\|\le (1+\beta)^N.$$

\end{proof}

%
%
%

\subsection{The plaque family theorem and Pliss Lemma}

For dominated splittings, we have local invariant center stable manifolds and local invariant center unstable manifolds \cite[Theorem 5.5]{HPS77}.

For $i\in\NN$, denote by $D^i(1)$ be the unit ball in $\RR^i$ and ${\rm Emb}(D^i(1),M)$ is the space of $C^1$ embeddings from $D^i(1)$ to $M$.

\begin{Lemma}\label{Lem:plaque}
Let $\Lambda$ be a compact invariant set with a dominated splitting $T_\Lambda M=E\oplus F$, where $\dim E=i$.  Then there is a neighborhood $\cal U$ of $f$ and a neighborhood $U$ of $\Lambda$ such that for any $g\in\cal U$, for any compact invariant set $\Lambda_g$ contained in $U$, denoting the dominated splitting of $\Lambda_g$ by $E_g\oplus F_g$, then there is a map $\Theta_g:\Lambda_g\to {\rm Emb}(D^i(1),M)$ such that when one denotes $W^{E_g}_\varepsilon(x,g)=\Theta_g(x)(D^i(\varepsilon))$, we have
\begin{itemize}

\item Invariance: for any $\varepsilon>0$, there is $\delta>0$ such that for any $x\in\Lambda_g$, we have $g(W^{E_g}_\delta(x,g))\subset W^{E_g}_\varepsilon(g(x),g)$.

\item Tangency: for any $x\in\Lambda_g$, we have $T_x W^{E_g}_\varepsilon(x,g)=E_g(x)$.

\item Continuity: when $g_n\to g$ as $n\to\infty$, $x_n\in\Lambda_{g_n}$ such that $x_n\to x\in\Lambda$, then $W^{E_{g_n}}(x_n,g_n)\to W^{E_g}(x,g)$.

\end{itemize}

We can also get the manifolds $\{W^F(x,g)\}_{x\in\Lambda_g}$ tangent to $F_g$.

\end{Lemma}

\begin{Remark}

We notice that the continuity is not stated in the original version of the plaque family theorem. From the proof of the plaque family theorem, one can know this property.

\end{Remark}

We have the following version of Pliss lemma \cite{Pli72} that is useful to get uniform estimations in some non-uniform setting. Recall that ${\rm m}(A)$ is the mini-norm of a linear isomorphism $A$, i.e., ${\rm m}(A)=\inf_{\|v\|=1}\|Av\|$.

\begin{Lemma}\label{Lem:pliss}

Assume that $\Lambda$ is a compact invariant set with a dominated splitting $T_\Lambda M=E\oplus F$. Given $N\in\NN$ and $\lambda_1>\lambda_2>1$ such that for any $x\in\Lambda$, if
$$\lim_{n\to\infty}\frac{1}{n}\sum_{i=0}^{n-1}\log {\rm m}(Df^N|_{F(f^{iN}(x))})\ge \lambda_1,$$
then there is a point $y$ in the positive orbit of $x$, we have for any $n\in\NN$
$$\frac{1}{n}\sum_{i=0}^{n-1}\log {\rm m}(Df^N|_{F(f^{iN}(y))})\ge \lambda_2.$$

\end{Lemma}

We have the following estimations on centre-unstable manifolds. The proof is a simple application of the mean value theorem, hence omitted.
\begin{Lemma}\label{Lem:expansion}
Assume that $\Lambda$ is a compact invariant set with a dominated splitting $T_\Lambda M=E\oplus F$. Given $n\in\NN$, for $\lambda_1>\lambda_2>1$, there are $C=C(\lambda_1,\lambda_2)$ and $\alpha_0=\alpha_0(\lambda_1,\lambda_2)$, for any $x\in\Lambda$ satisfying
$$\prod_{i=0}^{n-1}{\rm m}(Df^N|_{F(f^{iN}(y))})\ge \lambda_1^n,~~~\forall n\in\NN$$
for any $y$ and for any $n\in\NN$ such that
$$f^\ell(y)\in W^F_{\alpha_0}(f^\ell(x)),  ~~~\forall 0\le\ell \le n,$$
then we have
$$d(f^n(x),f^n(y))\ge C\lambda_2^n d(x,y).$$

\end{Lemma}

\subsection{The entropy of a plaque}

\begin{Lemma}\label{Lem:entropy-plaque}
Let $\Lambda$ be a compact invariant set that admits a dominated splitting $T_\Lambda M=E\oplus F$ without mixed behavior. For any $\varepsilon>0$, there is a neighborhood $\cal U$ of $f$ and a neighborhood $U$ of $\Lambda$ and $\alpha>0$ such that for any $g\in\cal U$ and for any point $x\in\Lambda_g\subset U$, we have 
$${\rm h}(g|_{W^E_\alpha(x)})\le\varepsilon,~~~{\rm h}(g^{-1}|_{W^F_\alpha(x)})\le\varepsilon.$$

\end{Lemma}

\begin{proof}

We only prove the case for $W^E$. The result for $W^F$ will be symmetric.

Given $\varepsilon>0$, we take $\beta>0$ such that $(\dim E)\log(1+2\beta)<\varepsilon.$ By Lemma~\ref{Lem:uniformity}, there is $N=N(\beta)\in\NN$ and a neighborhood $\cal U$ of $f$ and a neighborhood $U$ of $\Lambda$ such that for any $g\in\cal U$, for any compact invariant set $\Lambda_g$ of $g$ in $U$, we have that $\Lambda_g$ admits a dominated splitting 
$$T_{\Lambda_g}M=E_g\oplus F_g,$$

Now we have that for any $x\in\Lambda_g$,
$$\|Dg^{N}|_{E_g(x)}\|\le (1+\beta)^N.$$

Thus one can choose $\alpha>0$ small such that for any 
$z\in W^{E_g}_{\alpha}(x)$, we have
$$\|Dg^N|_{T_z W^{E_g}(x)}\|\le (1+2\beta)^N.$$ 
Thus, if we take $C=C(\beta)$ to be 
$$C=\max\{(1+2\beta)\|Df\|,(1+2\beta)^2\|Df^2\|,\cdots,(1+2\beta)^{N-1}\|Df^{N-1}\|\}+2.$$

Then we have for any $z\in W^{E_g}_{\alpha}(x)$ and any $n\in\NN$, if $g^{\ell}(z)\subset W^{E_g}_\alpha(g^\ell(x))$ for any $0\le \ell\le n-1$, then we have that
$$\|Dg^n|_{T_z W^{E_g}(x)}\|\le C(1+2\beta)^n.$$

Fix $\delta>0$. By using the Mean Value Theorem, for any $y,z\in W^{E_g}_{\alpha}(x)$ satisfying $d(y,z)<\delta$, when $f^\ell(y),f^{\ell}(z) \in W^{E_g}_{\alpha}(f^{\ell}x)$ for any $1\le\ell\le n-1$, there is $\xi_n\in W^{E_g}_{\alpha}(x)$ such that
$$d(g^n(y),g^n(z))\le \|Dg^n|_{T_{\xi_n} W^{E_g}(x)}\|d(y,z)\le C(1+2\beta)^n d(y,z).$$

Thus, the $n$-th Bowen ball $B_n(y,\delta)$ contains a ball of radius $\delta/C(1+2\beta)^n$. We consider the volume of the ball $B_n(y,\delta)$, then we have
$${\rm Volume}(B_n(y,\delta))\ge \frac{\delta^{\dim E}}{C^{\dim E}(1+2\beta)^{n \dim E}}.$$
Thus, there are at most
$$\left[\frac{C^{\dim E}\alpha^{\dim E}(1+2\beta)^{n\dim E}}{\delta^{\dim E}}\right]+1$$
disjoint $n$-th Bowen balls contained in $W^{E_g}_{\alpha}(x)$.
This implies the entropy is bounded by
$$\lim_{\delta\to0}\limsup_{n\to\infty}\frac{1}{n}\log\frac{C^{\dim E}\alpha^{\dim E}(1+2\beta)^{n\dim E}}{\delta^{\dim E}}\le (\dim E)\log(1+2\beta)<\varepsilon.$$

\end{proof}

\subsection{Estimation of the local entropy}\label{Sec:esti-localentropy}

We need the following lemma for local entropy.

\begin{Lemma}\label{Lem:local}

Assume that a compact invariant set $\Lambda$ admits a dominated splitting $T_\Lambda M=E\oplus F$ without mixed behavior. Then for any $\varepsilon>0$, there is $\alpha>0$ and a neighborhood ${\cal U}$ of $f$ and a neighborhood $U$ of $\Lambda$ such that for any $g\in \cal U$ and any compact invariant set $\Lambda_g\subset U$ of $g$, we have
$${\rm h}_{\alpha}(g|_{\Lambda_g})\le \varepsilon.$$

\end{Lemma}

\begin{proof}
We recall a result from \cite[Proposition 2.5]{LVY13}. For proving ${\rm h}_{\alpha}(g)\le\varepsilon$, it suffices to prove that for any ergodic invariant measure $\mu$ supported on $\Lambda_g$ of $g$, for $\mu$-a.e. $x$, for the bi-infinite Bowen ball, we have that
$${\rm h}(g,B_{\pm\infty}(x,\alpha))\le\varepsilon.$$
In fact, by Proposition~\ref{Pro:inverse-entropy}, it suffices to prove that
\begin{itemize}

\item either, ${\rm h}(g,B_{\pm\infty}(x,\alpha))\le\varepsilon$ for $\mu$-a.e. $x$;

\item or, ${\rm h}(g^{-1},B_{\pm\infty}(x,\alpha))\le\varepsilon$ for $\mu$-a.e. $x$

\end{itemize}

For the constants of the dominated splitting, we assume that there are $N\in\NN$ and $\lambda\in(0,1)$ (independent of $g$) such that for any $x\in\Lambda_g$, $\|Dg^N|_{E_g(x)}\|\|Dg^{-N}|_{F_g(g^N(x))}\|\le\lambda.$

We define the functions
$$\varphi^{E_g}(x)=\log\|Dg^N|_{E_g(x)}\|,~~~\psi^{F_g}(x)=\log{\rm m}(Dg^N|_{F_g(x)}).$$
$$S_n(\varphi^{E_g}(x))=\frac{1}{n}\sum_{i=0}^{n-1}\varphi^{E_g}(g^{iN}(x)),~~~S_n(\psi^{F_g}(x))=\frac{1}{n}\sum_{i=0}^{n-1}\psi^{F_g}(g^{iN}(x)).$$
By Birkhoff's ergodic theorem, the following two limits exist:
$$\lim_{n\to\infty}S_n(\varphi^{E_g}(x))=\int \varphi^{E_g}(x){\rm d}\mu,~~~\lim_{n\to\infty}S_n(\psi^{F_g}(x))=\int \varphi^{F_g}(x){\rm d}\mu.$$
By domination, at most one of the above quantities is contained in $(\log\lambda/2,-\log\lambda/2)$ for $\mu$-a.e. $x$.

Without loss of generality, one assume that $\lim_{n\to\infty}S_n(\psi^{F_g}(x))$ is not contained in this interval. Thus, we have that $\lim_{n\to\infty}S_n(\psi^{F_g}(x))\ge-\log\lambda/2$. In this case, we will prove that for $\mu$-a.e. $x$,  $B_{\pm\infty}(x,\alpha)\subset W^E_\alpha(x)$, and by applying Lemma~\ref{Lem:entropy-plaque}, one can conclude.

Notice that when $\lim_{n\to\infty}S_n(\psi^{E_g}(x))$ is not in this interval, then one can also prove that  for $\mu$-a.e. $x$,  $B_{\pm\infty}(x,\alpha)\subset W^F_\alpha(x)$. Then we need to apply Proposition~\ref{Pro:inverse-entropy} to prove that for $\mu$-a.e. $x$, we have that ${\rm h}(f^{-1},B_{\pm\infty}(x,\alpha))$ is small.


\bigskip

Take $C=C(\lambda^{-1/4},\lambda^{-1/5})$ and $\alpha_0=\alpha_0(\lambda^{-1/4},\lambda^{-1/5})$ as in Lemma~\ref{Lem:expansion}. By reducing $\alpha_0$ if necessary, one can assume that for any $w_1,w_2$ in some locally maximal invariant set of some neighborhood of $\Lambda$, if $d(w_1,w_2)<\alpha_0$, then
$$\lambda^{1/12}\le \frac{{\rm m}(Df^N|_{F(w_1)})}{{\rm m}(Df^N|_{F(w_2)})}\le \lambda^{-1/12}.$$
The above reduction implies
$$\lim_{n\to\infty}\frac{1}{n}\log\left(\prod_{i=0}^{n-1}{\rm m}(Df^N|_{F(f^{iN}(x))})\right)\ge -\frac{\log\lambda}{2}.$$
By Lemma~\ref{Lem:entropy-nStep}, it is enough to estimate the entropy at any iterate of $x$. By Lemma~\ref{Lem:pliss}, without loss of generality after an iteration, one can assume that
$$\prod_{i=0}^{n-1}{\rm m}(Df^N|_{F(f^{iN}(x))})\ge \lambda^{-n/3},~~~\forall n\in\NN.$$
By reducing $\alpha$ if necessary, since $y\in B_{\pm\infty}(x,\alpha)$, we have
$$\prod_{i=0}^{n-1}{\rm m}(Df^N|_{F(f^{iN}(y))})\ge \lambda^{-n/4},~~~\forall n\in\NN.$$
If there is $y\in B_{\pm\infty}(x,\alpha)\setminus W^{E_g}_\alpha(x)$, then we consider $z\in W^{F_g}_{\alpha_0}(y)\cap W^{E_g}_{\alpha_0}(x)$. There is $n_0$ such that

\begin{itemize}

\item such that $d(g^{n_0}(y),g^{n_0}(z))$ is almost $\alpha_0$ by Lemma~\ref{Lem:expansion}. This means that $n_0$ is related to $\alpha$: when $\alpha$ is small we have $n_0$ is large.

\item $d(g^{n_0}(x),g^{n_0}(z))/d(g^{n_0}(y),g^{n_0}(z))$ is small when $n_0$ is large by the domination. 

\item $d(g^{n_0}(x),g^{n_0}(y))$ is bounded by $\alpha$ since $y$ is contained in the Bowen ball of $x$ of size $\alpha$. 

\end{itemize}

When $\alpha\ll\alpha_0$, we have that $n_0$ is large. Thus,
$$d(g^{n_0}(y),g^{n_0}(z))>d(g^{n_0}(x),g^{n_0}(z))+d(g^{n_0}(x),g^{n_0}(y)).$$
Then one can get a contradiction by the triangle inequality.

%
%
%
%
%
%
%

\end{proof}

\begin{Definition}

For a compact metric space $X$ and a homeomorphism $T:X\to X$, $T$ is \emph{asymptotically entropy expansive} if for any $\varepsilon>0$, there is $\alpha>0$ such that for any $x\in X$, we have
$${\rm h}(B_\infty(x,\alpha))<\varepsilon.$$

\end{Definition}

We have the following corollary directly:

\begin{Corollary}\label{Cor:asymptotic}

Assume that a compact invariant set $\Lambda$ admits a dominated splitting $T_\Lambda M=E\oplus F$ without mixed behavior. Then $f|_\Lambda$ is asymptotically entropy expansive.

\end{Corollary}

Thus, we also have a ``principal symbolic extension''. 

\begin{Definition}

We say a compact invariant set $\Lambda$ admits a \emph{principal symbolic extension} if there is $n\in\NN$ and a compact invariant subset $\Sigma$ of the shift $(\{1,2,\cdots,n\}^\ZZ,\sigma)$, where $\sigma$ is the shift map, and a continuous surjective map $\pi:\Sigma\to\Lambda$ such that for any invariant measure $\mu$ of $(\Sigma,\sigma)$, the metric entropy of $\mu$ w.r.t. $\sigma$ is the same as the the metric entropy of $\pi_*(\mu)$ w.r.t. $f$.

\end{Definition}

It was proven by \cite{BFF02} that any asymptotically entropy expansive system admits a principal symbolic extension. Hence we have the following corollary directly:
\begin{Corollary}\label{Cor:asymptotic}

Assume that a compact invariant set $\Lambda$ admits a dominated splitting $T_\Lambda M=E\oplus F$ without mixed behavior. Then $\Lambda$ admits a principal symbolic extension.

\end{Corollary}
\subsection{Upper semi continuous of the metric entropy}
%
%
%
%
%
%
%
%

%
%
%

Now we can give the proof of Theorem~\ref{Thm:usc}.

\begin{proof}[Proof of Theorem~\ref{Thm:usc}]
Given a regular partition $\cal B$ of $\mu$, for any $\varepsilon>0$,  there is $n\in\NN$, for any $m\in\NN$ large enough, by using Lemma~\ref{Lem:finite-approximation}, we have

\begin{eqnarray*}
{\rm h}_{\mu}(f)&\ge &{\rm h}_{\mu}(f,{\cal B})\\
&\ge& \frac{1}{n}{\rm H}_{\mu}(\bigvee_{i=0}^{n-1}f^{-i}(\cal B))-\varepsilon\\
&\ge&-2\varepsilon +\frac{1}{n}{\rm H}_{\mu_m}(\bigvee_{i=0}^{n-1}f_m^{-i}(\cal B))\\
&\ge& -2\varepsilon+{\rm h}_{\mu_m}(f_m,{\cal B}).
\end{eqnarray*}

By \cite[Theorem 3.5]{Bow72}, we have that for any partition ${\cal B}$ whose norm is less than $\alpha$, we have
$${\rm h}_{\mu_m}(f_m|_{\Lambda_{f_m}})\le {\rm h}_{\mu_m}(f,{\cal B})+{\rm h}_{\alpha}({f_m|_{\Lambda_{f_m}}}).$$

By applying Lemma~\ref{Lem:local}, one can choose $\alpha>0$ such that ${\rm h}_{\alpha}({f_m|_{\Lambda_{f_m}}})<\varepsilon$ for $m$ large enough. Hence, by taking  an $\alpha$-regular partition ${\cal B}$, we have
\begin{eqnarray*}
{\rm h}_{\mu}(f)&\ge&  -2\varepsilon+{\rm h}_{\mu_m}(f_m,{\cal B})\ge {\rm h}_{\mu_m}(f_m|_{\Lambda_{f_m}})-{\rm h}_{\alpha}(f_m|_{\Lambda_{f_m}})-2\varepsilon\\
&\ge&{\rm h}_{\mu_m}(f_m)-3\varepsilon
\end{eqnarray*}
 for $m$ large enough. By taking a limit and by the arbitrariness of $\varepsilon$, one can get the conclusion.
\end{proof}

\section{The equilibrium state of $\psi(x)=-\log|\det Df|_{F(x)}|$}

In this section, we will consider a $C^1$ diffeomorphism $f$ that has a topological attractor with a dominated splitting $T_\Lambda M=E\oplus F$. We can extend the bundles $E$ and $F$ into a small neighborhood $U$ of $\Lambda$ continuously. The extensions are still denoted by $E$ and $F$. We can also extend the function $\psi(x)=-\log|\det Df|_{F(x)}|$ in a small neighborhood of $\Lambda$. In $U$, one can define the cone field ${\cal C}_\theta^F$ associated to $F$ of width $\theta>0$ by the following way:
$${{\cal C}_\theta^F}(x)=\{v=v^E+v^F\in T_x M:~|v^E|\le \theta|v^F|\}.$$

Since the splitting is dominated, the cone field ${\cal C}_\theta^F$ is positive invariant for some large iteration $Df^N$ and the width of $Df^n({\cal C}_\theta^F(x))$ tends to zero exponentially for some $x\in U$ by some uniform constants.

For the continuous function $\psi=-\log|\det Df|_{F}|$ and $n\in\NN$, define

$$S_n\psi(x)=\sum_{i=0}^{n-1}\psi(f^i(x)).$$

Some similar version of the following theorem has been already stated in \cite{LeY15}.  The proof based on volume estimation used in \cite{Qiu11}, originally from \cite{BoR75}.

\begin{Theorem}\label{Thm:es}

Let $\Lambda$ be a topological attractor which admits a dominated splitting $T_\Lambda M=E\oplus F$. Assume that the entropy function is upper semi continuous, then there is $\delta_0>0$ and $\theta>0$ such that for any manifold $D$ tangent to the cone ${\cal C}_{\theta}^F$, whose diameter is less than $\delta_0$, then for Lebesgue almost every point $x\in D$, for any accumulation point $\mu$ of
$$\{\frac{1}{n} \sum_{i=0}^{n-1}\delta_{f^i(x)}\},$$
we have
$${\rm h}_{\mu}(f)+\int \psi{\rm d}\mu\ge 0.$$
\end{Theorem}

By the properties of cone fields, there are $\theta>0$ and $r>0$ such that for any disc $D$ tangent to the cone field ${\cal C}_\theta^F$ and whose diameter is less than $r$, if the diameter of $f^n(D)$ is also less than $r$, then $f^n(D)$ tangent to the cone ${\cal C}_\theta^F$.

For $\varepsilon>0$, we consider the set of invariant measures:

$${\cal M}_\varepsilon=\{\mu:~{\rm h}_{\mu}(f)+\int \psi{\rm d}\mu\ge -\varepsilon.\}$$

Notice that we have

$${\cal M}_0=\bigcap_{n\in\NN}{\cal M}_{1/n}.$$

By the upper semi continuity of the metric entropy, we have that ${\cal M}_\varepsilon$ is closed. Thus ${\cal M}\setminus{\cal M}_\varepsilon$ is open. Thus in the metric space of invariant measures, there are countably many open sets $\{{\cal O}_i\}_{i\in\NN}$ such that

\begin{itemize}

\item the union of all ${\cal O}_i$ is ${\cal M}\setminus{\cal M}_\varepsilon.$

\item Each ${{\cal O}}_i$ is convex and open.

\item the closure of ${{\cal O}}_i$ is contained in ${\cal M}\setminus{\cal M}_\varepsilon$.

\end{itemize}

For each set ${\cal O}$, we define
$$B_D({\cal O})=\{x\in D:~\{\frac{1}{n}\sum_{i=0}^{n-1}\delta_{f^i(x)}\}\textrm{~has an accumulation point in~}{\cal O}\}.$$
$$B_D({\cal O},n)=\{x\in D:~\frac{1}{n}\sum_{i=0}^{n-1}\delta_{f^i(x)}\in {\cal O}\}$$
From the definition, we have
$$B_D({\cal O})\subset\limsup_{n\to\infty}B_n({\cal O},n)=\bigcap_{n\ge 1}\bigcup_{m\ge n}B_D({\cal O},m)$$

We have the following result to conclude Theorem~\ref{Thm:es}.

\begin{Lemma}\label{Lem:zero-Leb}

For each ${\cal O}\in\{{\cal O}_i\}$, we have that the Lebesgue measure of $B_D({\cal O})$ is zero.

\end{Lemma}

\begin{proof}[Proof of Theorem~\ref{Thm:es}]
By Lemma~\ref{Lem:zero-Leb}, for any small $C^1$ sub manifold $D$ tangent to the cone field ${\cal C}_\theta^F$, we have that 
Lebesgue almost every point $x$ in $D$, any accumulation point $\mu$ of $\{\frac{1}{n} \sum_{i=0}^{n-1}\delta_{f^i(x)}\},$ we have that ${\rm h}_{\mu}(f)+\int \psi{\rm d}\mu\ge 0.$ A small neighborhood of $\Lambda$ can be foliated by such kind of sub manifolds. Thus, the proof can be complete.
\end{proof}

Now we give the proof of Lemma~\ref{Lem:zero-Leb}.

\begin{proof}[Proof of Lemma~\ref{Lem:zero-Leb}]

By using the Borel-Cantelli argument, for proving ${\rm Leb}(B_D({\cal O}))=0$, it suffices to prove that
$$\sum_{n=1}^\infty {\rm Leb}(B_D({\cal O},n))<\infty.$$
Thus we need to estimate ${\rm Leb}(B_D({\cal O},n))$ for $n$ large enough.

We consider
$$B_D(\overline{\cal O},n)=\{x\in D:~\frac{1}{n}\sum_{i=0}^{n-1}\delta_{f^i(x)}\in \overline{\cal O}\}$$

We first cover $B_D(\overline{\cal O},n)$ by a maximal $(n,\delta)$-seperated set $\Delta_{n,\delta}$. Since it is maximal, we have
$$B_D({\cal O},n)\subset B_D(\overline{\cal O},n)\subset \bigcup_{x\in\Delta_{n,\delta}}B_n(x,\delta).$$

We need to choose two constants. Notice that by positive iterations, the cone ${\cal C}^F_\theta$ will decrease exponentially. Thus, by considering a positive iteration of $D$ (saying $f^{N}(D)$) and then dividing the positive iteration into small pieces, one can assume that $D$ is tangent to a very thin cone field (since $f^N(D)$ is tangent to a very thin cone field).

We can choose constants $C_{\delta}$ such that for any disc $W$ tangent to the cone field ${\cal C}^F_\theta$, for any points $x,y\in W$ satisfying $d_W(x,y)<\delta$, we have $|\psi(x)-\psi(y)|\le\log{C_{\delta}}$. By the uniform continuity of $\psi$, one can assume that $C_\delta\to 1$ as $\delta\to 0$.

For any $\kappa>0$, there is $\theta_\kappa$ such that for any disc $W$ tangent to the cone fied ${\cal C}^F_{\theta_\kappa}$, we have for any $x\in W$,
$$\left| \log |{\rm det}Df|_{T_x W}|-\log\psi(x)\right|<\kappa.$$

There is $N_\kappa\in\NN$ such that for any $n>N_\kappa$, for any sub-manifold $W$ tangent to ${\cal C}^F_\theta$, then $f^n(W)$ is tangent to ${\cal C}^F_{\theta_\kappa}$.

Thus, there is $C_\kappa$ (large) such that

\begin{eqnarray*}
{\rm Leb}(B_D(\overline{\cal O},n)) &\le& \sum_{x\in\Delta_{n,\delta}}{\rm Leb}B_n(x,\delta)=\sum_{x\in\Delta_{n,\delta}}\int_{B_n(x,\delta)}{\rm d}{\rm Leb}_D(y)\\
&=& \sum_{x\in\Delta_{n,\delta}}\int_{f^n(B_n(x,\delta))} \prod_{i=0}^n |\det (Df|_{T_{f^{-n+i}(z)}f^i W})|^{-1} {\rm d}{\rm Leb}_{f^n D}(z)\\
&\le& C_\kappa {\rm e}^{n\kappa}  \sum_{x\in\Delta_{n,\delta}}\int_{f^n(B_n(x,\delta))}  {\rm e}^{S_n\psi(z)} {\rm d}{\rm Leb}_{f^n D}(z)\\
&\le&  V_\delta C_\kappa {\rm e}^{n\kappa} C_\delta^n\sum_{x\in\Delta_{n,\delta}}{\rm e}^{S_n\psi(x)},
\end{eqnarray*}
where $V_\delta$ is the maximal volume of a disc $D$ whose diameter is less than $\delta$, which is tangent to ${\cal C}^F_\theta$.

Now we need to estimate $\sum_{x\in\Delta_{n,\varepsilon}}{\rm e}^{S_n\psi(x)}$. Take $\nu_n$ and $\mu_n$:

$$\nu_n=\frac{\sum_{x\in \Delta_{n,\varepsilon}} {\rm e}^{S_n\psi(x)}\delta_x}{\sum_{x\in \Delta_{n,\varepsilon}} {\rm e}^{S_n\psi(x)}}$$

$$\mu_n=\frac{1}{n}\sum_{i=0}^{n-1}f^i_*\nu_n=\sum_{x\in\Delta_{n,\varepsilon}}\frac{{\rm e}^{S_n\psi(x)}}{\sum_{x\in \Delta_{n,\varepsilon}} {\rm e}^{S_n\psi(x)}}\frac{1}{n}\sum_{i=0}^{n-1}\delta_{f^i(x)}.$$

\begin{Claim}

$\mu_n\in \overline{\cal O}$.

\end{Claim}

\begin{proof}[Proof of the Claim]

Since $x\in\Delta_{n,\varepsilon}\subset B_D(\overline{\cal O},n)$, we have that
$$\frac{1}{n}\sum_{i=0}^{n-1}\delta_{f^i(x)}\in \overline{\cal O}.$$
By the convexity of $\overline{\cal O}$, the claim is true.

\end{proof}

We have that any accumulation point $\mu$ of $\{\mu_n\}$ is invariant. And moreover $\mu\in\overline {\cal O}$. By the construction of $\overline{\cal O}$, we have ${\rm h}_\mu(f)+\int\psi{\rm d}\mu\le-\varepsilon$. 

Now we want to prove
$${\rm h}_{\mu}(f)+\int \psi{\rm d}\mu\ge \limsup_{n\to\infty} \frac{1}{n} \log\sum_{x\in\Delta_{n,\delta}}{\rm e}^{S_n\psi(x)}.$$

Take a partition ${\cal B}=\{B_1,B_2,\cdots,B_k\}$ that is $\delta$-regular for $\mu$. Then we have that every element of $\bigvee_{i=0}^{n-1}f^{-i}{\cal B}$ contains at most one point in $\Delta_{n,\delta}$.

By \cite[Chapter 9]{Wal82}, we have
 
 $${\rm H}_{\nu_n}(\bigvee_{i=0}^{n-1}f^{-i}{\cal B})+\int S_n\psi {\rm d}\nu_n=\log\sum_{x\in\Delta_{n,\delta}}{\rm e}^{S_n\psi(x)}.$$

Now we need to consider the relationship between ${\rm H}_{\mu_n}$ and ${\rm H}_{\nu_n}$.

Given some integer $1\le j<q<n$,  the partition $\bigvee_{i=0}^{n-1}f^{-i}{\cal B}$ can be written in the following way:

$$\bigvee_{i=0}^{n-1}f^{-i}{\cal B}=\bigvee_{r=0}^{[(n-j)/q]-1}f^{-(rq+j)} \bigvee_{i=0}^{q-1}f^{-i}{\cal B} \vee \bigvee_{\ell\in S_j}f^{-\ell}{\cal B},$$
where $S_j=\{0,1,\cdots,j-1,j+[(n-j)/q]q,\cdots,n-1\}$. We have $|S_j|\le 2q$, and

\begin{eqnarray*}
\log\sum_{x\in\Delta_{n,\delta}}{\rm e}^{S_n\psi(x)} &=& {\rm H}_{\nu_n}(\bigvee_{j=0}^{n-1}f^{-j}{\cal B})+\int S_n\psi {\rm d}\nu_n\\
&\le& \sum_{r=0}^{[(n-j)/q]-1}{\rm H}_{\nu_n}(f^{-(rq+j)}\bigvee_{i=0}^{q-1}f^{-i}{\cal B})+{\rm H}_{\nu_n}(\bigvee_{k\in S_j}f^{-k}({\cal B}))+\int S_n\psi{\rm d}\nu_n\\
&\le&  \sum_{r=0}^{[(n-j)/q]-1}{\rm H}_{f^{rq+j}_*\nu_n}(\bigvee_{i=0}^{q-1}f^{-i}{\cal B})+2q\log k+\int S_n\psi{\rm d}\nu_n
\end{eqnarray*}

By taking the sum for $j$ from 0 to $q-1$, and using Lemma~\ref{Lem:H}, we have

\begin{eqnarray*}
q\log\sum_{x\in\Delta_{n,\delta}}{\rm e}^{S_n\psi(x)} &\le& \sum_{p=j}^{j+[(n-j)/q]q}{\rm H}_{f^p_*\nu_n}(\bigvee_{i=0}^{q-1}f^{-i}{\cal B})+2q^2\log k+q\int S_n\psi{\rm d}\nu_n\\
&\le& \sum_{p=0}^{n-1}{\rm H}_{f^p_*\nu_n}(\bigvee_{i=0}^{q-1}f^{-i}{\cal B})+2q^2\log k+q\int S_n\psi{\rm d} \nu_n\\
&\le& n{\rm H}_{\mu_n}(\bigvee_{i=0}^{q-1}f^{-i}{\cal B})+2q^2\log k+q\int S_n\psi{\rm d}\nu_n.  
\end{eqnarray*}

By dividing by $n$, we have
$$\frac{q}{n} \log\sum_{x\in\Delta_{n,\varepsilon}}{\rm e}^{S_n\psi(x)} \le {\rm H}_{\mu_n}(\bigvee_{i=0}^{q-1}f^{-i}{\cal B})+\frac{2q^2\log k}{n}+q\int \psi{\rm d}\mu_n.$$
By taking the $\limsup$ for $n$ we have
$$q\limsup_{n\to\infty}\frac{1}{n} \log\sum_{x\in\Delta_{n,\delta}}{\rm e}^{S_n\psi(x)}\le {\rm H}_{\mu}(\bigvee_{i=0}^{q-1}f^{-i}{\cal B})+q\int \psi{\rm d}\mu.$$
By dividing $q$ and letting $q\to\infty$, we have
$$\limsup_{n\to\infty}\frac{1}{n} \log\sum_{x\in\Delta_{n,\delta}}{\rm e}^{S_n\psi(x)}\le {\rm h}_{\mu}({\cal B})+\int \psi {\rm d}\mu.$$

Recall that ${\rm h}_{\mu}(f)+\int \psi {\rm d}\mu\le-\varepsilon$, we have 
$$\limsup_{n\to\infty}\frac{1}{n} \log\sum_{x\in\Delta_{n,\delta}}{\rm e}^{S_n\psi(x)}\le -\varepsilon.$$

Thus,
\begin{eqnarray*}
\limsup_{n\to\infty}\frac{1}{n}\log {\rm Leb}(B_D({\cal O},n)) &\le& \limsup_{n\to\infty}\frac{1}{n}\log V_\delta C_\kappa+\kappa+\log C_\delta\\
&+& \lim_{n\to\infty}\frac{1}{n}\log\sum_{x\in\Delta_{n,\delta}}{\rm e}^{S_n\psi(x)}\\
&\le&    \kappa+\log C_\delta-\varepsilon.                                                                  
\end{eqnarray*}

By choosing $\kappa>0$ small and $C_\delta$ close to $1$, we have that 
$$\limsup_{n\to\infty}\frac{1}{n}\log {\rm Leb}(B_D({\cal O},n))<0.$$
Then by using the Borel-Cantelli argument, we can complete the proof.
\end{proof}

\begin{proof}[Proof of the main theorem]

Now we assume that $\Lambda$ is a topological attractor that admits a dominated splitting $T_\Lambda M=E\oplus F$ without mixed behavior. Notice that the entropy function is upper semi continuous by Corollary~\ref{Cor:usc}. Then by Theorem~\ref{Thm:es} we have that there is a measure $\mu$ such that
$${\rm h}_\mu(f)\ge \int \log |{\rm Det} Df|_F|{\rm d}\mu.$$
Since there is no mixed behavior, we have that 
$${\rm h}_\mu(f)\ge \int \sum\lambda_+{\rm d}\mu,$$
Where $\sum\lambda^+$ is the sum of positive Lyapunov exponents of $\mu$. On the other hand, by Ruelle's inequality, we have
$${\rm h}_\mu(f)\le \int \sum\lambda_+{\rm d}\mu.$$
Thus $\mu$ satisfies the entropy formula.

\end{proof}

\vskip 5pt

\noindent Dawei Yang

\noindent School of Mathematical Sciences

\noindent Soochow University, Suzhou, 215006, P.R. China

\noindent yangdw1981@gmail.com, yangdw@suda.edu.cn

\vskip 5pt

\noindent Yongluo Cao

\noindent School of Mathematical Sciences

\noindent Soochow University, Suzhou, 215006, P.R. China

\noindent ylcao@suda.edu.cn

\end{document}